\documentclass[11pt,oneside]{article}
\usepackage{amsthm,amsfonts,amsmath,amssymb}
\usepackage[colorlinks=true,linkcolor=magenta,citecolor=blue,urlcolor=cyan,pdfstartview=FitB]{hyperref}
\usepackage[left=1.1in,right=1.1in,top=1.1in,bottom=1.1in]{geometry}
\usepackage{natbib}
\usepackage{hyphenat}
\usepackage{footmisc}
\usepackage[doublespacing]{setspace}
\usepackage{hyphenat}
\newtheorem{theorem}{Theorem}[]
\newtheorem{corollary}{Corollary}[]

\newtheorem{definition}{Definition}[]
\newtheorem{lem}{Lemma}[section]

\newcommand{\mbE}{\mathbb{E}}

\newcommand{\rbk}[1]{\left(#1\right)}
\newcommand{\sbk}[1]{\left[#1\right]}
\newcommand{\cbk}[1]{\left\{#1\right\}}

\newcommand{\mc}{\mathcal}

\newcommand{\tPinull}{\tilde{\pi}_{0}\left(\lambda\right)}

\newcommand{\ThreshRej}{t_{\alpha }\left( \widetilde{FDR}_{\lambda }\right)}
\newcommand{\RejT}{\tilde{t}_{\alpha }\left(\lambda\right)}
\newcommand{\Gfdr}{\widetilde{FDR}_{\lambda }}

\begin{document}

\title{\vspace{-1cm}Stopping time property of thresholds of Storey-type FDR procedures}
\author{Xiongzhi Chen\footnote{Corresponding author. Center for Statistics and Machine Learning, and Lewis-Sigler Institute for
Integrative Genomics, Princeton University, Princeton, NJ 08544. Email:
\texttt{xiongzhi@princeton.edu}.}
\ and Rebecca W. Doerge\footnote{Departments of Statistics and Agronomy,
Purdue University, West Lafayette, IN 49707. Email: \texttt{doerge@purdue.edu.}}
}
\date{}
\maketitle

\begin{abstract}
For multiple testing, we introduce Storey-type FDR procedures and
the concept of ``regular estimator of the proportion of true nulls''. We show that the rejection threshold of a Storey-type FDR procedure is a stopping time with respect to the backward filtration generated by the p-values and that a Storey-type FDR estimator at this rejection threshold equals the pre-specified FDR level, when the estimator of the
proportion of true nulls is regular. These results hold regardless
of the dependence among or the types of distributions of the p-values. They directly imply that
a Storey-type FDR procedure is conservative when the null p-values are independent and uniformly distributed.
\end{abstract}


\section{Introduction}

To control false discovery rate (FDR, \citealp{Benjamini:1995}) in multiple testing, it is crucial to show the conservativeness of an FDR procedure, i.e., that the FDR of an FDR procedure is no larger than a pre-specified level. To show the conservativeness of their FDR procedures,
\cite{Storey:2004} used the stopping time property (STP) of the rejection threshold that was claimed in their Lemma 4, and \cite{Liang:2012} quoted this lemma as their Lemma 5 and used it to prove their Theorem 7. They assumed that the p-values corresponding to the true null hypotheses (i.e., null p-values) are independent and uniformly distributed. However, none of them provided a formal proof of the STP of the rejection thresholds, and without the STP the optional stopping theorem (see, e.g., \citealp{Karatzas:1991}) can not be applied and the martingale based approach in \cite{Storey:2004} and \cite{Liang:2012} fails.

In this article, we introduce Storey-type FDR procedure in \autoref{Sec:review} as an extension of Storey's procedure in \cite{Storey:2004}, introduce the concept of regular estimator of the proportion of true nulls, and show in \autoref{Sec:STproperty} that the STP of the rejection threshold of a Storey-type FDR procedure is generic
when the estimator of the proportion of true nulls is regular. This provides a formal justification of the use of STP of the rejection thresholds in \cite{Storey:2004} and, if needed, in \cite{Chen:2012GenEst}. However, it also implies that the rejection
threshold of an adaptive FDR procedure may not be a stopping time when the estimator of proportion true nulls is not regular.
Further, we show that a Storey-type FDR estimator at the threshold of its corresponding Storey-type FDR procedure equals to the targeted FDR level. Some consequences of these findings are given in \autoref{sec:cons}, which include that a Storey-type FDR procedure is always conservative when the null p-values are independent and uniformly distributed. We end the article with a short discussion in \autoref{SecDiscussion}.

\section{Storey-type FDR procedures}\label{Sec:review}

Let there be $m$ null hypotheses $H_{i}$ with associated p-values $p_{i}$ for $i=1,...,m$, such that only $m_{0}$ among them are true nulls and the rest false nulls. However, the true status of each $H_i$ is unknown, and so is the proportion of true nulls $\pi _{0}=m_{0}/m$. A one-step multiple testing procedure (MTP) claims that $H_i$ is a false null if and only if $p_i \leq t$ using a rejection threshold $t\in \left[ 0,1\right] $. It gives $R\rbk{t}$ as the total number of null hypotheses claimed to be false and $V\rbk{t}$ as the number of true nulls claimed to be false. The FDR \citep{Benjamini:1995} of such an MTP is defined as
\begin{equation*}
FDR\left( t\right) = \mbE\sbk{V\rbk{t}/\max\cbk{R\rbk{t},1}}\text{,}
\end{equation*}
where $\mbE$ is the expectation.

Let $\tPinull$ with a tuning parameter $ \lambda\in\lbrack0,1)$ be an estimator of
$\pi_{0}$ such that $\tPinull$ ranges in $\left[0,1\right]  $. We define a ``Storey-type FDR estimator'' as
\begin{equation}\label{eqGP}
\widetilde{FDR}_{\lambda }\left( t\right) =\min \left\{ 1,\dfrac{
\tPinull t}{m^{-1}\max \cbk{ R\left(
t\right),1} }\right\} \text{\quad for \quad } t \in \sbk{0,1}.
\end{equation}
For $\alpha \in \sbk{0,1}$, let
\begin{equation}\label{eqThresh}
t_{\alpha}\left(\widetilde{FDR}_{\lambda }\right) =\sup \left\{ t\in \left[ 0,1\right] :
\widetilde{FDR}_{\lambda }\left( t\right) \leq \alpha \right\}
\end{equation}
and the decision rule based on $\widetilde{FDR}_{\lambda }$
\begin{equation}\label{eqRejG}
\text{claim }  H_i \text{ as a false null } \Longleftrightarrow p_i \leq t_{\alpha}\left(\widetilde{FDR}_{\lambda }\right).
\end{equation}
The procedure defined by \eqref{eqGP} and \eqref{eqRejG} is called a ``Storey-type
FDR procedure''.

A Storey-type FDR procedure allows the use of various estimators of $\pi_0$ and is very versatile. For example, (i) it is Storey's procedure when $\tPinull$ in \eqref{eqGP} is the estimator
\begin{equation}\label{eqpi0A}
\hat{\pi}_{0}\left( \lambda \right) = \min\left\{1,
 \left( 1-\lambda \right)^{-1}m^{-1}
\sum\nolimits_{i=1}^{m}\mathbf{1}_{\left\{ p_{i}>\lambda \right\} } \right\}
\end{equation}
in \cite{Storey:2004}; (ii) it is the generalized FDR procedure
in \cite{Chen:2012GenEst}, designed for multiple testing based on discrete p-values, when $\tPinull$ is the estimator of $\pi_0$ proposed there and further studied in \cite{Chen:2015discretefdr}; (iii) it is the BH procedure in \cite{Benjamini:1995} when $\lambda=0$ is set in \eqref{eqpi0A} in Storey's procedure or $\tPinull \equiv 1 $ is set;
(iv) it is a ``dynamic adaptive FDR procedure'' studied in \cite{Liang:2012} when $\lambda$ in \eqref{eqpi0A} is determined from the p-values and can be a random variable.

\section{Stopping time property of the rejection threshold}\label{Sec:STproperty}

For notational simplicity, we will from now on write $\ThreshRej$ as $\RejT$.
Define the backward filtration
\begin{equation}\label{eq:backFil}
  \mathcal{F}_{t}=\sigma \left( \mathbf{1}_{\left\{ p_{i}\leq s\right\}
},t\leq s\leq 1,i=1,...,m\right)
\end{equation}
for each $t \in [0,1)$ 
and the ``stopped backward'' filtration
$\mathcal{G}=\left\{ \mathcal{F}_{t\wedge \lambda }:1\geq t\geq 0\right\}$. Further, introduce the following

\begin{definition}\label{def:regpi0}
An estimator $\tPinull$ of the proportion $\pi_0$ of true nulls with tuning parameter $\lambda
\in\lbrack0,1)$ is said to be regular if it is measurable with respect to (wrt) $\mathcal{F}_{\lambda}$.
\end{definition}

It is crucial to note that, if $\lambda$ is a functional of the p-values such that the information contained in $\mathcal{F}_{\lambda}$ is not sufficient to determine the value of $\tPinull$,
then $\tPinull$ is not measurable wrt $\mathcal{F}_{\lambda}$. This can happen for dynamic Storey procedures considered in
\cite{Liang:2012}, where $\lambda$ is adaptively determined from the data. When this happens, $\RejT$ does not have to be a stopping time wrt to $\mathcal{G}$, the optional stopping theorem can not be applied, and the martingale
arguments to prove the conservativeness of (dynamic) Storey FDR procedure in \cite{Storey:2004} and \cite{Liang:2012} are invalid.
Here is out main result:
\begin{theorem}
\label{ThmStoppingTime}
If $\tPinull$ is regular, then $\RejT$ is a stopping time with respect to $\mathcal{G}$. Further, if $\RejT < 1$, then
\begin{equation}\label{eqExhuast}
\widetilde{FDR}_{\lambda
}\left( \tilde{t}_{\alpha }\left(\lambda\right) \right)
=\alpha.
\end{equation}
\end{theorem}

The proof of \autoref{ThmStoppingTime} is given in \autoref{app}.
Note that $\RejT < 1$ excludes the meaningless case $\alpha = 1$ as the targeted FDR level. Therefore, $\RejT < 1$
in \autoref{ThmStoppingTime} essentially is not an assumption. \autoref{ThmStoppingTime} shows that, regardless
of whether the p-values are independent, whether they are continuously distributed, or the stochastic orders
of their distributions wrt the standard uniform random variable, the rejection threshold of a Storey-type FDR procedure
is a stopping time wrt to the stopped backward filtration, when the estimator of the proportion of true nulls is regular.
It also shows that a Storey-type
FDR estimator equals the pre-specified FDR level at the threshold of its corresponding FDR procedure, when
the estimator of the proportion of true nulls is regular and the corresponding decision rule is statistically meaningfully implemented. Equally importantly, \autoref{ThmStoppingTime} shows that the STP may not hold for the rejection threshold
of a Storey-type procedure when $\tPinull$ is not regular, i.e., when adaptivity of the estimator $\tPinull$ to $\mathcal{F}_{\lambda}$ is not ensured.

We make four remarks on \autoref{ThmStoppingTime} and its proof:
(i) the measurability of $\tPinull$ wrt to $\mathcal{F}_{\lambda}$ is critical to the STP of $\RejT$;
(ii) the fact that \eqref{eqExhuast}, i.e., $\widetilde{FDR}_{\lambda
}\left( \tilde{t}_{\alpha }\left(\lambda\right) \right)
=\alpha$ for $\RejT <1$, holds for general p-values is a new result, whose consequences will be discussed in \autoref{sec:cons}; (iii) the validity of the STP of $\RejT$, regardless of the joint distribution of the p-values, depends crucially on the
intrinsic structure of the backward filtration $\mathcal{G}$;
(iv) the STP of the threshold (wrt a similar backward filtration) defined by equation (7.2) in \cite{Pena:2011} does hold trivially due to the special structure of the filtration there.

\section{Conservativeness of Storey-type FDR procedures}\label{sec:cons}

We give a few implications of \autoref{ThmStoppingTime}. For concise statements, we set assumption A1) as ``the null p-values are independent and uniformly distributed''. Further, for each $t \in [0,1)$ let
$\mathcal{H}_{t}=\sigma \left( \mathbf{1}_{\left\{ p_{i}\leq s\right\}},0\leq s\leq t,i=1,...,m\right)$.
\begin{corollary}\label{cor1}
  Assume A1), the following hold:
  \begin{enumerate}
    \item Lemma 4 of \cite{Storey:2004} is valid when $\hat{\pi}_{0}\rbk{\lambda}$ in \eqref{eqpi0A} is regular.
    \item Storey's procedure in \cite{Storey:2004} is conservative when $\hat{\pi}_{0}\rbk{\lambda}$ in \eqref{eqpi0A} is regular and $\RejT < 1$.
    \item If $\lambda$ is a stopping time wrt to $\cbk{\mc{H}_t: t\in [0,1)}$ and lies in a compact subset $\sbk{\kappa,\tau}$ in $\rbk{0,1}$ for some $\kappa < \tau$ and $\tPinull$ is regular, then a Storey-type FDR procedure is conservative when $\tilde{t}_{\alpha } \leq \kappa$.
  \end{enumerate}
\end{corollary}
\begin{proof}
  The first claim is a direct consequence of \autoref{ThmStoppingTime}. We show the second claim. By Corollary 1 in \cite{Liang:2012}, $\mbE\sbk{\widetilde{FDR}_{\lambda }\left( \tilde{t}_{\alpha }\right)} \geq
  FDR\rbk{\tilde{t}_{\alpha }}$. However, \autoref{ThmStoppingTime} implies $\mbE\sbk{\widetilde{FDR}_{\lambda }\left( \tilde{t}_{\alpha }\right)}= \alpha$. So, $FDR\rbk{\tilde{t}_{\alpha }} \leq \alpha$. Finally, we show the third.
   By Theorem 4 in \cite{Liang:2012}, $\mbE\sbk{\widetilde{FDR}_{\lambda }\left( \tilde{t}_{\alpha }\right)} \geq
  FDR\rbk{\tilde{t}_{\alpha }}$. Therefore, from \autoref{ThmStoppingTime} we see $FDR\rbk{\tilde{t}_{\alpha }} \leq \alpha$. This completes the proof.
\end{proof}

\autoref{cor1} illustrates the connection between the regularity of the estimator of the proportion of true nulls, the STP of the rejection threshold, and the conservativeness of a Storey-type FDR procedure. The third conclusion in \autoref{cor1} allows to choose $\lambda$ in $\tPinull$ from the data in a Storey-type FDR procedure without potentially sacrificing its conservativeness. Further, it covers the dynamic Storey's FDR procedure in \cite{Liang:2012} where $\hat{\pi}_{0}\rbk{\lambda}$ in \eqref{eqpi0A}
is used and $\lambda$ is determined from the data.

Compared to the proofs of the conservativeness of (dynamic) Storey's procedure in \cite{Storey:2004} and \cite{Liang:2012}, the proof of \autoref{cor1} we provide based on \autoref{ThmStoppingTime} is much shorter and the second conditioning step there is no longer needed. This illustrates the usefulness of our finding \eqref{eqExhuast} for general p-values in the martingale arguments to prove the conservativeness of a Storey-type FDR procedure.

\section{Discussion}\label{SecDiscussion}

We have shown the STP of the rejection threshold of a Storey-type FDR procedure for general p-values, that a Storey-type FDR
estimator equals the targeted FDR level at its rejection threshold, and that a Storey-type FDR procedure is usually conservative, when the estimator of the proportion of true nulls is regular.
This implies that the STP may not hold for the rejection threshold of an adaptive FDR procedure when an estimator of the proportion of true nulls is not adaptive to the backward filtration. Therefore, caution should be taken before applying martingale arguments to prove the conservativeness of an adaptive FDR procedure.
In view of the increasing use of stochastic processes (such as decision processes) and suprema related to these processes (such as rejection thresholds) in multiple testing, to better understand
the behavior of these suprema remains an important task.

\section*{Acknowledgements}

This research was funded by a National Science Foundation Plant Genome
Research Program grant (No. IOS-1025976) to RWD. Its major part
was completed when X. Chen was a PhD student at the Department of Statistics, Purdue University.

\numberwithin{equation}{section}
\appendix
\section{Proofs}\label{app}

In order to prove the STP of the rejection threshold, we first need to understand the
the \textit{scaled
inverse rejection process} $L\left( t\right) =t\left( \max\cbk{R\left( t\right),
1}\right)^{-1}$ with $t\in \left[ 0,1\right] $, with which the regularity of $\tPinull$,
the property of the stopped backward filtration $\mc{G}$ and
a contrapositive argument will yield \autoref{ThmStoppingTime}.
In what follows, no assumption will be made about the independence
between the p-values, the continuity of the p-value distributions, or their stochastic orders wrt
the standard uniform distribution.

\subsection{Downward jumps of the scaled inverse rejection process}\label{sec:inverseRej}

Order the p-values into $p_{\left( 1\right)
}<p_{\left( 2\right) }<...<p_{\left( n\right) }$ distinctly, where the
multiplicity of $p_{\left( i\right) }$ is $n_{i}$ for $i=1,...,n$. Let $%
p_{\left( n+1\right) }=\max \left\{ p_{\left( n\right) },1\right\} $ and $%
p_{\left( 0\right) }=0.$ Define $T_{j}=\sum\limits_{l=1}^{j}n_{l}$ for $j=1,...,n$.

\begin{lem}
\label{lemmaratioprocess}$\bigskip $The process $\left\{ L\left( t\right)
,t\in \lbrack 0,1]\right\} $ is such that%
\begin{equation}
L\left( t\right) =\left\{
\begin{array}{lcl}
t & \text{if} & t\in \lbrack 0,p_{\left( 1\right) })\text{,} \\
tT_{j}^{-1} & \text{if} & t\in \lbrack p_{\left( j\right) },p_{\left(
j+1\right) })\text{ for }j=1,..,n-1\text{,} \\
tm^{-1} & \text{if} & t\in \lbrack p_{\left( n\right) },p_{\left ( n+1\right)}]\text{.}%
\end{array}%
\right.  \label{3}
\end{equation}%
Moreover, it can only be discontinuous at $p_{\left( i\right) }$, $1\leq
i\leq n$, where it can only have a downward jump with size%
\begin{equation*}
L\left( p_{\left( i\right) }-\right) -L\left( p_{\left( i\right) }\right) =%
\dfrac{p_{\left( i\right) }n_{i}}{R\left( p_{\left( i\right) }\right)
\left[ R\left( p_{\left( i\right) }\right) -n_{i}\right] }>0\text{.}
\end{equation*}
\end{lem}
\begin{proof}
  Clearly%
\begin{equation*}
R\left( t\right) =\left\{
\begin{array}{ccl}
0 & \text{if} & 0\leq t<p_{\left( 1\right) }\text{,} \\
T_{j} & \text{if} & p_{\left( j\right) }\leq t<p_{\left( j+1\right)
},j=1,..,n-1\text{,} \\
m & \text{if} & p_{\left( n\right) }\leq t\leq p_{\left( n+1\right) }\text{,}%
\end{array}%
\right.
\end{equation*}%
and (\ref{3}) holds. Therefore, the points of discontinuities of $L\left(
\cdot \right) $ are the original distinct p-values. This justifies the
first part of the assertion.

Now we show that $L\left( \cdot \right) $ can only have downward jumps at
points of discontinuity. Define%
\begin{equation*}
\varphi \left( t,\eta \right) =\dfrac{t+\eta }{R\left( t+\eta \right) }-%
\dfrac{t}{R\left( t\right) }=\dfrac{\eta R\left( t\right) +t\left[ R\left(
t\right) -R\left( t+\eta \right) \right] }{R\left( t+\eta \right) R\left(
t\right) }\text{.}
\end{equation*}%
From the fact
$
R\left( p_{\left( j\right) }\right) -R\left( p_{\left( j\right) }-\right)
=n_{j}>0
$
but $R\left( p_{\left( j\right) }+\right) -R\left( p_{\left( j\right)
}\right) =0$ for each $1\leq j\leq n,$ it follows that $\varphi \left(
p_{\left( j\right) },0+\right) =\lim_{\eta \downarrow 0}\varphi \left(
t,\eta \right) =0$ but%
\begin{equation*}
\varphi \left( p_{\left( j\right) },0-\right) =\lim_{\eta \uparrow 0}\varphi
\left( p_{\left( j\right) },\eta \right) =\dfrac{p_{\left( j\right) }n_{j}}{%
R\left( p_{\left( j\right) }\right) \left[ R\left( p_{\left( j\right)
}\right) -n_{j}\right] }>0\text{.}
\end{equation*}%
Thus $L\left( p_{\left( j\right) }-\right) -L\left( p_{\left( j\right)
}\right) =\varphi \left( p_{\left( j\right) },0-\right) >0$ and the proof is
completed.
\end{proof}
\autoref{lemmaratioprocess} shows that the process $L\left(t\right)$ is piecewise linear and can only have downward jumps as $t$
increases.
The conclusion of \autoref{lemmaratioprocess} is right the contrary to the claim in the proof
of Theorem 2 in \cite{Storey:2004} that \textquotedblleft the process $%
mt/R\left( t\right) $ has only upward jumps and has a final value of $1$%
\textquotedblright , since it says \textquotedblleft the process $mt/R\left(
t\right) $ has only downward jumps\textquotedblright . We construct a
counterexample to their claim as follows.

For a small increase $c$ in $t$
which results an increase $a_{c}$ in $R\left( t\right) $, we see that%
\begin{equation*}
L\left(t+c\right) - L\left(t\right)= \dfrac{t+c}{R\left( t\right) +a_{c}}-\dfrac{t}{R\left( t\right) }=\dfrac{%
cR\left( t\right) -ta_{c}}{\left[ R\left( t\right) +a_{c}\right] R\left(
t\right) }<0
\end{equation*}%
if and only if $\dfrac{c}{a_{c}}<\dfrac{t}{R\left( t\right) }$. Construct $m$
p-values with $n\geq 4$ such that there exists some $1\leq j_{0}<n-2$ with
$n_{j_{0}+1}>T_{j_{0}}$ but $p_{\left( j_{0}+1\right) }<1$. Choose $c_{1}$
and $c_{2}$ such that $0<c_{1}<\left( p_{\left( j_{0}+1\right) }-p_{\left(
j_{0}\right) }\right) /2$ and $0<c_{2}<p_{\left( j_{0}+1\right) }-2c_{1}$.
Let $t_{0}=p_{\left( j_{0}+1\right) }-c_{1}$ and $c=c_{1}+c_{2}$. Then $%
p_{\left( j_{0}\right) }<t_{0}<p_{\left( j_{0}+1\right) }$, $R\left(
t_{0}\right) =T_{j_{0}}$ and $R\left( t_{0}+c\right) =T_{j_{0}}+n_{j_{0}+1}$%
. Further, $0<c<t_{0}$ and $a_{c}=n_{j_{0}+1}$. So $\dfrac{c}{n_{j_{0}+1}}<%
\dfrac{t_{0}}{T_{j_{0}}}$ and $L\left( t_{0}+c\right) -L\left( t_{0}\right)
<0$. Letting $c\rightarrow 0$ gives $p_{\left( j_{0}+1\right) }$ as a point
of downward jump for $L\left( t\right) $.

\subsection{Proof of \autoref{ThmStoppingTime}}

When $\alpha = 0$, $\RejT=0$ if $\tPinull >0$
but $\RejT=1$ if $\tPinull =0$. In this case, $\RejT$ is already a stopping time
and $\Gfdr\rbk{\RejT} = \alpha = 0$.
On the other hand, $\RejT=1$ when $\alpha=1$. In this case, $\RejT$
is a stopping time and $\Gfdr\rbk{\RejT} = \tPinull$.
Note that $\tPinull=0$ can not happen when $0< \RejT <1$.

First, we show the claim on the stopping time property. By the previous discussion,
we only need to consider the case where $0<\RejT<1$ and $\tPinull >0$.
Define%
\begin{equation*}
\mathbf{X}_{t}^{\left( m\right) }\left( \omega \right) =\left( \mathbf{1}_{\left\{
p_{1}\leq t\right\} }\left( \omega \right) ,...,\mathbf{1}_{\left\{ p_{m}\leq
t\right\} }\left( \omega \right) \right) \text{, }\omega \in \Omega \text{,}
\end{equation*}%
where $\left(\Omega,\mathcal{A},\mathbb{P}\right)$ is the probability space
with $\Omega$ the sample space, $\mathcal{A}$ the sigma-algebra, and
$\mathbb{P}$ the probability measure.
Then $\mathcal{F}_{t}=\sigma \left( \mathbf{X}_{s}^{\left( m\right) }\left(
\omega \right) ,1\geq s\geq t\right) $, $t\in \left[ 0,1\right] $ and $%
\left\{ \mathcal{F}_{t}:0\leq t\leq 1\right\} $ is a non-increasing sequence of
sub-sigma-algebras of $\mathcal{A}$. Write $\widetilde{FDR}_{\lambda }\left(
t\right) =\widetilde{FDR}_{\lambda }\left( t,\omega \right) $.

By the definition of $\RejT$,%
\begin{equation}\label{eq:TheSet}
\left\{ \omega \in \Omega :\RejT\leq s\right\}
=\bigcap\limits_{\left\{ t:s<t\leq 1\right\} }A_{t}=\tilde{A}_{s}\text{,}
\end{equation}%
where $A_{t}=\left\{ \omega \in \Omega :\widetilde{FDR}_{\lambda }\left(
t,\omega \right) >\alpha \right\} $, $1\geq t>s$. Thus, we only need to show $\tilde{A}%
_{s}\in \mathcal{F}_{s\wedge \lambda }$.
Since either $\mathcal{F}_{s\wedge \lambda }=\mathcal{F}_{s}\supseteq
\mathcal{F}_{\lambda }$ when $s\leq \lambda $ or $\mathcal{F}_{s\wedge
\lambda }=\mathcal{F}_{\lambda }\supseteq \mathcal{F}_{s}$ when $s\geq
\lambda $, the stopping time property holds once we prove $\tilde{A}_{s}\in
\mathcal{F}_{s}$. Let $\mathbb{Q}$ be the set of all rational numbers.
Since $0< \alpha <1$ and $\tPinull >0$, the following decompositions are valid:
$
A_{t}=\bigcup\nolimits_{r\in \mathbb{Q}}\left( A_{t,r}\cap B_{r}\right)
$
and%
\begin{equation*}
\tilde{A}_{s}=\bigcap\limits_{\left\{ t:s<t\leq 1\right\}
}\bigcup\limits_{r\in \mathbb{Q}}\left( A_{t,r}\cap B_{r}\right)
=\bigcup\limits_{r\in \mathbb{Q}}\bigcap\limits_{\left\{ t:s<t\leq
1\right\} }\left( A_{t,r}\cap B_{r}\right) \text{,}
\end{equation*}%
where $A_{t,r}=\left\{ \omega \in \Omega :\dfrac{t}{m^{-1}\max\cbk{ R\left(
t\right),1} }\geq r\right\} $ and $B_{r}=\left\{ \omega \in
\Omega :\dfrac{\alpha }{\tPinull }<r\right\} $. Thus, it suffices to show%
\begin{equation}\label{eq:SetEqualA}
\bigcap\limits_{\left\{ t:s<t\leq 1\right\} }\left( A_{t,r}\cap
B_{r}\right) =\bigcap\limits_{\left\{ t:s<t\leq 1\right\} }A_{t,r}\cap
B_{r}\in \mathcal{F}_{s}\text{,}
\end{equation}

We now move to show
\begin{equation}\label{eq:SetEqualB}
A_{s,r}^{\ast }=\bigcap\limits_{\left\{ t:s<t\leq 1\right\} }A_{t,r}\in
\mathcal{F}_{s}\text{.}
\end{equation}
Define $I_{i}=[p_{\left( i\right) },p_{\left( i+1\right) })$, $i=1,...,n-1$.
We will add $I_{0}=[p_{\left( 0\right) },p_{\left( 1\right) })$ and $%
I_{n}=[p_{\left( n\right) },p_{\left( n+1\right) }]$ when $p_{\left(
n\right) }<1$. When $p_{\left( n\right) }=1$, we take $I_{n-1}$ to be $\left[
p_{\left( n-1\right) },p_{\left( n\right) }\right] $. Obviously there must
be a unique $j^{\ast }$ with $0\leq j^{\ast }\leq n$ such that $s\in
I_{j^{\ast }}$. Given $R\left( 1\right) =m$ and $\tPinull
 \in \left[ 0,1\right] $, the properties
of $L\left( \cdot \right) $ in \autoref{lemmaratioprocess} imply%
\begin{equation*}
A_{s,r}^{\ast }=A_{s,r}\bigcap \left( \bigcap\nolimits_{j=j^{\ast
}+1}^{n+1}A_{p_{\left( j\right) },r}\right) \text{.}
\end{equation*}%
Consequently, $A_{s,r}^{\ast }\in \mathcal{F}_{s}$, i.e., \eqref{eq:SetEqualB} holds.

Recall $\tilde{A}_{s}$ defined in \eqref{eq:TheSet}. If $B_{r}\in \mathcal{F}_{\lambda }$, then \eqref{eq:SetEqualB} implies \eqref{eq:SetEqualA}, which implies $\tilde{A}_{s}\in \mathcal{F}_{s}$, i.e., $\RejT$ is a STP wrt to $\mc{G}$ if $\tPinull$ is measurable wrt to $\mc{F}_{\lambda}$. However, this holds since $\tPinull$ is regular.

Finally, we show $\widetilde{FDR}_{\lambda }\left( \tilde{t}_{\alpha }\right) =\alpha
$, where we have written $\RejT$ as $\tilde{t}_{\alpha }$.
Clearly, $\widetilde{FDR}_{\lambda }\left( \tilde{t}_{\alpha }\right) \leq \alpha$
by the definition of $\tilde{t}_{\alpha }$. Our arguments next proceed by contrapositive
reasoning, i.e., that
$\widetilde{FDR}_{\lambda }\left( \tilde{t}_{\alpha }\right) < \alpha$ gives a contradiction.
Obviously, there must be a unique $0\leq j^{\prime }\leq n$ such that $\tilde{t}%
_{\alpha }\in I_{j^{\prime }}$.
Since $\tilde{t}_{\alpha }<1$, we have $\alpha <1$, $\tPinull >0$, $p_{\left(j^{\prime}+1\right)}>0$ and
$I^{\ast }=\left[\tilde{t}_{\alpha },%
p_{\left(j^{\prime}+1\right)}\right)\subseteq $ $I_{j^{\prime }}$.
Let $\rho _{m}^{\ast }\left( t\right) =m^{-1}L\left( t\right) $. Then
$\rho _{m}^{\ast }\left( t\right)$ is
continuous and strictly increasing on $I_{j^{\prime }}$ by \autoref{lemmaratioprocess}.
If
$\widetilde{FDR}_{\lambda }\left( \tilde{t}_{\alpha }\right) < \alpha$ and $\alpha <1$,
then there must be some $d^{\prime}>0$ such that
$\tilde{I}^{\ast}=\left[\tilde{t}_{\alpha }, \tilde{t}_{\alpha }+d^{\prime}\right] \subseteq I^{\ast }$
and that
\begin{equation}\label{eqReduction}
\widetilde{FDR}_{\lambda }\left( t\right) = \tPinull \rho _{m}^{\ast }\left( t\right)
\text{\quad for all \quad} t \in \tilde{I}^{\ast}.
\end{equation}
However, \eqref{eqReduction} implies that there exists
$\hat{t}_{\alpha }\in \tilde{I}^{\ast }$ such that $\hat{t}_{\alpha }>\tilde{t}%
_{\alpha }$ but $$\widetilde{FDR}_{\lambda }\left(\hat{t}_{\alpha }\right)
= \tPinull \rho _{m}^{\ast }\left( \hat{t}_{\alpha }\right) <\alpha.$$
This contradicts the definition of $\tilde{t}_{\alpha }.$
Hence $\widetilde{FDR}_{\lambda }\left( \tilde{t}_{\alpha }\right) =\alpha
$ must hold, which completes the whole proof.


\bibliographystyle{chicago}

\begin{thebibliography}{}

\bibitem[\protect\citeauthoryear{{Benjamini} and {Hochberg}}{{Benjamini} and
  {Hochberg}}{1995}]{Benjamini:1995}
{Benjamini}, Y. and Y.~{Hochberg} (1995).
\newblock Controlling the false discovery rate: a practical and powerful
  approach to multiple testing.
\newblock {\em J. R. Statist. Soc. Ser. B\/}~{\em 57\/}(1), 289--300.

\bibitem[\protect\citeauthoryear{{Chen} and {Doerge}}{{Chen} and
  {Doerge}}{2014}]{Chen:2012GenEst}
{Chen}, X. and R.~{Doerge} (2014).
\newblock Generalized estimators formultiple testing: proportion of true nulls
  and false discovery rate.
\newblock {\em \url{http://arxiv.org/abs/1410.4274}\/}.

\bibitem[\protect\citeauthoryear{{Chen} and {Doerge}}{{Chen} and
  {Doerge}}{2015}]{Chen:2015discretefdr}
{Chen}, X. and R.~{Doerge} (2015).
\newblock A weighted fdr procedure under discrete and heterogeneous null
  distributions.
\newblock {\em \url{http://arxiv.org/abs/1502.00973v2}\/}.

\bibitem[\protect\citeauthoryear{{Karatzas} and {Shreve}}{{Karatzas} and
  {Shreve}}{1991}]{Karatzas:1991}
{Karatzas}, I. and S.~E. {Shreve} (1991).
\newblock {\em Brownian Motion and Stochastic Calculus\/} (2 ed.).
\newblock Graduate Texts in Mathematics. Springer Verlag.

\bibitem[\protect\citeauthoryear{Liang and Nettleton}{Liang and
  Nettleton}{2012}]{Liang:2012}
Liang, K. and D.~Nettleton (2012).
\newblock Adaptive and dynamic adaptive procedures for false discovery rate
  control and estimation.
\newblock {\em J. R. Statist. Soc. Ser. B\/}~{\em 74\/}(1), 163--182.

\bibitem[\protect\citeauthoryear{{Pena}, {Habiger}, and {Wu}}{{Pena}
  et~al.}{2011}]{Pena:2011}
{Pena}, E.~A., J.~D. {Habiger}, and W.~{Wu} (2011).
\newblock Power-enhanced multiple decision functions controlling family-wise
  error and false discovery rates.
\newblock {\em Ann. Statist.\/}~{\em 39\/}(1), 556--583.

\bibitem[\protect\citeauthoryear{{Storey}, {Taylor}, and {Siegmund}}{{Storey}
  et~al.}{2004}]{Storey:2004}
{Storey}, J.~D., J.~E. {Taylor}, and D.~{Siegmund} (2004).
\newblock Strong control, conservative point estimation in simultaneous
  conservative consistency of false discover rates: a unified approach.
\newblock {\em J. R. Statist. Soc. Ser. B\/}~{\em 66\/}(1), 187--205.

\end{thebibliography}

\end{document}